\documentclass[11pt,a4paper]{amsart}
\usepackage{amsthm}
\usepackage{amssymb}
\usepackage{latexsym}
\usepackage{amscd}
\usepackage{graphics}

\newcommand{\bprop} {\begin{proposition}}
\newcommand{\eprop} {\end{proposition}}
\newcommand{\btheo} {\begin{theorem}}
\newcommand{\etheo} {\end{theorem}}
\newcommand{\blem} {\begin{lemma}}
\newcommand{\elem} {\end{lemma}}
\newcommand{\bcor} {\begin{corollary}}
\newcommand{\ecor} {\end{corollary}}

\newcommand{\Be}{\begin{equation}}
\newcommand{\Ee}{\end{equation}}
\newcommand{\Bea}{\begin{eqnarray}}
\newcommand{\Eea}{\end{eqnarray}}
\newcommand{\Bes}{\begin{equation*}}
\newcommand{\Ees}{\end{equation*}}
\newcommand{\Beas}{\begin{eqnarray*}}
\newcommand{\Eeas}{\end{eqnarray*}}
\newcommand{\Ba}{\begin{array}}
\newcommand{\Ea}{\end{array}}

\newtheorem{theorem}{T{\hskip 0pt\footnotesize\bf HEOREM}}[section]
\newtheorem{lemma}[theorem]{L{\hskip 0pt\footnotesize\bf EMMA}}
\newtheorem{proposition}[theorem]{P{\hskip 0pt\footnotesize\bf ROPOSITION}}

\newtheorem{corollary}[theorem]{C{\hskip 0pt\footnotesize\bf OROLLARY}}

\begin{document}


\title[Off-diagonal weighted estimates of the Bergman projection]{Sharp off-diagonal weighted norm estimates for the Bergman projection\\ (Not for publication)}
\author{Beno\^it F. Sehba}
\address{Department of Mathematics University of Ghana PO. Box LG 62 Legon Accra Ghana}


\email{bfsehba@ug.edu.gh}


\begin{abstract}
We prove that for $1<p\le q<\infty$, $qp\geq {p'}^2$ or $p'q'\geq q^2$, $\frac{1}{p}+\frac{1}{p'}=\frac{1}{q}+\frac{1}{q'}=1$, $$\|\omega P_\alpha(f)\|_{L^p(\mathcal{H},y^{\alpha+(2+\alpha)(\frac{q}{p}-1)}dxdy)}\le C_{p,q,\alpha}[\omega]_{B_{p,q,\alpha}}^{(\frac{1}{p'}+\frac{1}{q})\max\{1,\frac{p'}{q}\}}\|\omega f\|_{L^p(\mathcal{H},y^{\alpha}dxdy)}$$
where $P_\alpha$ is the weighted Bergman projection of the upper-half plane $\mathcal{H}$, and $$[\omega]_{B_{p,q,\alpha}}:=\sup_{I\subset \mathbb{R}}\left(\frac{1}{|I|^{2+\alpha}}\int_{Q_I}\omega^{q}dV_\alpha\right)\left(\frac{1}{|I|^{2+\alpha}}\int_{Q_I}\omega^{-p'}dV_\alpha\right)^{\frac{q}{p'}},$$
with $Q_I=\{z=x+iy\in \mathbb{C}: x\in I, 0<y<|I|\}$.
\end{abstract}

\keywords{ Bergman projection, B\'ekoll\'e-Bonami weight}
\subjclass[2010] {Primary: 47B38, 30H20, 47B35 Secondary: 42C40, 42A61}


\maketitle


\section{Introduction and statement of the results}
Let $\mathcal{H}$ be the upper-half plane, that is the set $\{z=x+iy\in \mathbb {C}:y>0\}$. We denote by $L_\alpha^p(\mathcal H)$ the Lebesgue space $L^p(\mathcal H, y^\alpha dxdy)$, that is the space of all functions $f$ such that 
\begin{equation}\label{eq:bergnormdef}
\|f\|_{p,\alpha}^p:=\int_{\mathcal H}|f(x+iy)|^py^\alpha dxdy<\infty.
\end{equation}

For $\alpha>-1$ and $1<p<\infty$, the weighted Bergman space $A_\alpha^p(\mathcal H)$ is the subspace of $L_\alpha^p(\mathcal H)$ consisting of analytic functions.
We recall that the Bergman space $A_{\alpha}^{2}(\mathcal H)$ ($-1< \alpha<\infty$) is a reproducing kernel Hilbert space with kernel $K_w^\alpha (z)=K^\alpha (z,w)=\frac{1}{(z-\overline {w})^{2+\alpha}}$. That is for any $f\in A_{\alpha}^{2}(\mathcal H)$, the following representation holds
\begin{equation}
f(w)=P_{\alpha}f(w)=\langle f,K_w^{\alpha}\rangle_{\alpha}=\int_{\mathcal {H}}f(z)K^{\alpha}(w,z)dV_{\alpha}(z),
\end{equation}
where for simplicity, we write $dV_\alpha(x+iy)=y^\alpha dxdy$. We write $P_\alpha^+$ for the positive Bergman operator which is defined by replacing $K_w^{\alpha}$ by $|K_w^{\alpha}|$ in the definition of $P_\alpha$. Note that the boundedness of $P_\alpha^+$ implies the boundedness of $P_\alpha$.
\vskip .3cm

Let $I$ be an interval in $\mathbb{R}$, we denote by $Q_I$ the set $$Q_I=\{z=x+iy\in \mathbb{C}: x\in I, 0<y<|I|\}.$$ Let $\omega$ be a positive locally integrable function defined on $\mathcal H$, $1<p<\infty$, $pp'=p+p'$,  and $\alpha>-1$. We say $\omega$ is a B\'ekoll\'e-Bonami weight (or $\omega$ belongs to the class $B_{p,\alpha}(\mathcal H)$) if the following quantity denoted $[\omega]_{B_{p,\alpha}}$, is finite
$$\sup_{I\subset \mathbb R,\,\,\, I\,\,\,\textrm{interval}}\left(\frac{1}{|I|^{2+\alpha}}\int_{Q_I}\omega(z)dV_\alpha(z)\right)\left(\frac{1}{|I|^{2+\alpha}}\int_{Q_I}\omega(z)^{1-p'}dV_\alpha(z)\right)^{p-1}.$$

\vskip .3cm
D. B\'ekoll\'e and A. Bonami proved in \cite{Bek,BB} that the Bergman projection is bounded on $L^p(\mathcal{H}, \omega dV_\alpha)$ if and only if the weight $\omega$ is in the class $B_{p,\alpha}(\mathcal H)$.  We are interested in this note in the off-diagonal version of their result. For this, it seems natural to first understand the classical situation, that is when only powers of the distance to the boundary are involved. We focus only on the positive operator and consider at this stage the following more general family
$$T^{+}f(z)=\int_{\mathcal{H}}\frac{f(w)}{|z-\bar{w}|^{1+b}}(\Im w)^a dV(w).$$
We have the following result.
\begin{theorem}\label{thm:offmain}

Suppose $\alpha,\beta>-1$, $1< p\le q<\infty$. Then the
following conditions are equivalent:

\begin{itemize}
\item[(a)]

The operator $T^+$ is bounded from
$L_{\alpha}^{p}(\mathcal{H})$ into $L_{\beta}^{q}(\mathcal{H})$.

\item[(b)]

The parameters satisfy 
\begin{equation}\label{eq:relationalphabetagamma}
b =a+1-\frac{\alpha+2}{p}+\frac{\beta+2}{q}
\end{equation} 
and
\begin{equation}\label{eq:condineq1}
\alpha+1<p(a+1).
\end{equation}
\end{itemize}
\end{theorem} 
The proof of the necessary part in the above theorem follows as in \cite{BanSeh}, the proof of the sufficient part can be obtained using the off-diagonal Schur test due to O. Okikiolu \cite{Oki}. As a  particular case of the above result, we have the following.
\begin{corollary}\label{cor:offmain}
Let $1<p\le q<\infty$, and $\alpha>-1$. Then the positive Bergman projection $P_\alpha^+$ is bounded from 
$L^p(\mathcal{H},y^\alpha dxdy)$ into\\ $L^q(\mathcal{H},y^{\alpha+(2+\alpha)(\frac{q}{p}-1)}dxdy)$.
\end{corollary} 

It follows from the last result that the right question on the off-diagonal version of the B\'ekoll\'e-Bonami result is the following: for which pairs of weights $\omega,\nu$ the operator $P_\alpha^+$ is bounded from  $L^p(\mathcal{H},\omega(x+iy) y^\alpha dxdy)$ into\\ $L^q(\mathcal{H},\nu(x+iy) y^{\alpha+(2+\alpha)(\frac{q}{p}-1)}dxdy)$?
\vskip .3cm
Before giving our answer to the above question, let us remark that quite recently, S. Pott and M. C. Reguera \cite{PR} have obtained the sharp bound (in the case $p=q$) of the Bergman projection in terms of the characteristic $[\omega]_{B_{p,\alpha}}$. More precisely, they proved the following. 

\begin{theorem}[S. Pott and M. C. Reguera \cite{PR}]\label{thm:pott}
Let $1<p<\infty$, and $-1< \alpha<\infty$. Suppose that $\omega\in B_{p,\alpha}(\mathcal H) $. Then
$P_\alpha^+$ is bounded on $L^p(\omega dV_\alpha)$.
Moreover, 
\Be\label{eq:bergnormestim}\|P_\alpha^+\|_{L^p(\omega dV_\alpha)\rightarrow L^p(\omega dV_\alpha)}\le C(p)[\omega]_{B_{p,\alpha}}^{\max\{1,\frac{p'}{p}\}}
\Ee
and the above estimate is sharp.
\end{theorem}
Let $\omega$ be a weight function and let $1<p\le q<\infty$. We say $\omega$ belongs to the class $B_{p,q,\alpha}$ if the following quantity is finite
$$[\omega]_{B_{p,q,\alpha}}:=\sup_{I\subset \mathbb{R}}\left(\frac{1}{|I|^{2+\alpha}}\int_{Q_I}\omega^{q}dV_\alpha\right)\left(\frac{1}{|I|^{2+\alpha}}\int_{Q_I}\omega^{-p'}dV_\alpha\right)^{\frac{q}{p'}}.$$
The class $B_{p,q,\alpha}$ is the analogue of the class introduced by B. Muckenhoupt and R. Wheeden in \cite{MuckWheed} and which is the range of weights for which the fractional operator (Riesz potential) is bounded.
\vskip .2cm
Let us state our one weight answer to the above question.
\begin{theorem}\label{thm:Bergpq}
Let $1<p\leq q<\infty$, so that $qp\geq {p'}^2$ or $p'q'\geq q^2$, and let $\alpha>-1$. Then $P_\alpha^+$ is bounded from $L^p(\mathcal{H}, \omega^pdV_\alpha)$ into $L^q(\mathcal{H}, \omega^qdV_{\alpha+(2+\alpha)(\frac{q}{p}-1)})$ if and only $\omega\in B_{p,q,\alpha}(\mathcal{H})$. Moreover,
$$\|\omega P_\alpha^+(f)\|_{L^p(\mathcal{H},y^{\alpha+(2+\alpha)(\frac{q}{p}-1)}dxdy)}\le C_{p,q,\alpha}[\omega]_{B_{p,q}}^{(\frac{1}{p'}+\frac{1}{q})\max\{1,\frac{p'}{q}\}}\|\omega f\|_{L^p(\mathcal{H},y^{\alpha}dxdy)},$$
and the power of $[\omega]_{B_{p,q,\alpha}}$ in the above inequality is sharp.

\end{theorem}
We remark that the boundedness of $P_\alpha^+$ from $L^p(\mathcal{H}, \omega^pdV_\alpha)$ into $L^q(\mathcal{H}, \omega^qdV_{\alpha+(2+\alpha)(\frac{q}{p}-1)})$ is equivalent to the boundedness from $L^p(\mathcal{H}, \omega^pdV_\alpha)$ into $L^q(\mathcal{H}, \omega^qdV_\alpha)$ of the operator
$$(\Im z)^{(2+\alpha)(\frac{1}{p}-\frac{1}{q})}\int_{\mathcal {H}}f(z)|K^{\alpha}(w,z)|dV_{\alpha}(z).$$
We are then led to consider the more general question of the boundedness  from $L^p(\mathcal{H}, \omega^pdV_\alpha)$ into $L^q(\mathcal{H}, \omega^qdV_\alpha)$ of the operator
 $$S_{\alpha,a}f(z)=(\Im z)^{a}\int_{\mathcal {H}}f(z)|K^{\alpha}(w,z)|dV_{\alpha}(z)$$
where $0\leq a<2+\alpha$. We have the following result.
\begin{theorem}\label{thm:Sbetapq}
Let $\alpha>-1$, and $0\leq a<2+\alpha$. Let $1<p\leq q<\infty$, so that $\frac{1}{p}-\frac{1}{q}=\frac{a}{2+\alpha}$ and $\min\{\frac{p'}{q},\frac{q}{p'}\}\leq 1-\frac{a}{2+\alpha}$. Then $S_{\alpha,a}$ is bounded from $L^p(\mathcal{H}, \omega^pdV_\alpha)$ into $L^q(\mathcal{H}, \omega^qdV_\alpha)$ if and only if $\omega\in B_{p,q,\alpha}(\mathcal{H})$. Moreover,
$$\|S_{\alpha,a}\|_{L^p(\mathcal{H},\omega^pdV_\alpha)\rightarrow L^q(\mathcal{H},\omega^qdV_\alpha)}\le C_{p,q,\alpha}[\omega]_{B_{p,q,\alpha}}^{(1-\frac{a}{2+\alpha})\max\{1,\frac{p'}{q}\}},$$
and the power of $[\omega]_{B_{p,q,\alpha}}$ in the above inequality is sharp.

\end{theorem}
Note that the sufficent part and the norm estimate in Theorem \ref{thm:Bergpq} follow from Theorem \ref{thm:Sbetapq}. The proof of the necessary part in Theorem \ref{thm:Bergpq} follows from usual arguments (see the proof of Theorem \ref{thm:Sbetapq}). We observe that Theorem \ref{thm:Bergpq} generalizes Theorem \ref{thm:pott} as the latter just corresponds to the case $p=q$ in our result.
\vskip .3cm
As in \cite{PR}, we follow the current trend in dyadic harmonic analysis, we use dyadic models for the operator $S_{\alpha,a}$ and make use of methods similar to the ones in \cite{Laceyetal,Moen}. As said above, for the proof of Theorem \ref{thm:offmain}, we will make use of an off-diagonal Schur-type test from \cite{Oki}. The main difficulty with the latter resides in the choice of the test functions and the other parameters of the test. 
\vskip .3cm
Our presentation is as follows: Theorem \ref{thm:Sbetapq} is proved in Section 2. In Section 3, we give an example that shows that our estimates are sharp. We make a remark about a sibling of the operator $S_{\alpha,a}$ in Section 4. Finally, for completeness, we give a proof of Theorem \ref{thm:offmain} in the last section.
\vskip .3cm
As usual, given two positive quantities $A$ and $B$, the notation $A\lesssim B$ (resp. $A\gtrsim B$)  means that $A\le CB$ (resp. $B\le CA$) for some absolute positive constant $C$. The notation $A\backsimeq B$ mean that $A\lesssim B$ and $B\lesssim A$. We will use $C_p$ or $C(p)$ to say that the constant $C$ depends only on $p$.

\section{Proof of Theorem \ref{thm:Sbetapq}}

Let us start by recalling some notions and notations. We consider the following system of dyadic grids,
$$\mathcal D^\beta:=\{2^j\left([0,1)+m+(-1)^j\beta\right):m\in \mathbb Z,\,\,\,j\in \mathbb Z \},\,\,\,\textrm{for}\,\,\,\beta\in \{0,1/3\}.$$
We also consider the following positive operators.
\begin{equation}\label{eq:discretoper}
Q_{\alpha,a}^\beta f:=\sum_{I\in \mathcal {D}^\beta}|I|^a\langle f,\frac{{\bf 1}_{Q_I}}{|I|^{2+\alpha}}\rangle_\alpha {\bf 1}_{Q_I}.
\end{equation}
By comparing the positive kernel $$K_\alpha^+(z,w)=\frac{1}{|z-w|^{2+\alpha}}$$ and the box-type kernel
$$K_\alpha^\beta(z,w):=\sum_{I\in \mathcal {D}^\beta}\frac{{\bf 1}_{Q_I}(z){\bf 1}_{Q_I}(w)}{|I|^{2+\alpha}},$$
one obtains the following (see \cite{PR} for the case $a=0$).
\begin{proposition}\label{prop:compareP+Pbeta}
Let $\alpha>-1$. Then there exists a constant $C>0$ such that for any $f\in L_{loc}^1{(\mathcal H,dV_\alpha)}$, $f\ge 0$, and $z\in \mathcal H$,
\begin{equation}\label{eq:compareP+beta}
S_{\alpha,a}f(z)\le C\sum_{\beta\in \{0,1/3\}}Q_{\alpha,a}^\beta f(z).
\end{equation}
\end{proposition}
For any weight $\omega$, and any subset $E\subset \mathcal{H}$, we write $$|E|_{\omega,\alpha}:=\int_E \omega dV_\alpha.$$
Given the dyadic grid $\mathcal D^\beta$ and a positive weight $\omega$, we define the fractional dyadic maximal function $M_{\omega,\alpha,a}^\beta $ for any $f\in L_{loc}^1(\mathcal H, dV_\alpha)$ by
$$\mathcal{M}_{\omega,\alpha,a}^\beta f=\sup_{I\in \mathcal {D}^\beta}\frac{{\bf 1}_{Q_I}}{|Q_I|_{\omega,\alpha}^{1-\frac{a}{2+\alpha}}}\int_{Q_I}|f|\omega dV_\alpha.$$
When $\omega=1$, we simply write $M_{\alpha,a}^\beta$ for the unweighted fractional maximal function.
Following for example the techniques in \cite{Moen}, one obtains the following.
\begin{proposition}\label{prop:fractmax}
If $0\leq a<2+\alpha$, $1<p<\frac{2+\alpha}{a}$ and $\frac{1}{q}=\frac{1}{p}-\frac{a}{2+\alpha}$, then 
$$\|\mathcal{M}_{\omega,\alpha,a}^\beta f\|_{L^q(\mathcal{H},\omega dV_\alpha)}\le C_{a,\alpha}\left(1+\frac{p'}{q}\right)^{1-\frac{a}{2+\alpha}}\|f\|_{L^p(\mathcal{H},\omega dV_\alpha)}.$$
\end{proposition}
\begin{proof}[Proof of Theorem \ref{thm:Sbetapq}]
We start by proving the sufficient part. We recall that given $Q_I$, its upper-half is the set $$T_I:=\{x+iy\in \mathcal H: x\in I,\,\,\,\textrm{and}\,\,\,\frac{|I|}{2}<y<|I|\}.$$ It is clear that if $\mathcal D$ is a dyadic grid in $\mathbb R$, then the family $\{T_I\}_{I\in \mathcal D}$ forms a tiling of $\mathcal H$.

Now observe that to prove Theorem \ref{thm:Sbetapq}, it is enough by Proposition \ref{prop:compareP+Pbeta} to prove that the following boundedness holds (with the right estimate of the norm)
\begin{equation}\label{eq:boundednessQbeta}
Q_{\alpha,a}^{\beta}: L^p(\omega^p dV_\alpha)\rightarrow L^q(\omega^q dV_\alpha),\,\,\,\beta\in \{0,1/3\}.
\end{equation}
The latter is equivalent to the following
\begin{equation}\label{eq:boundednessQbetasigma}
Q_{\alpha,a}^{\beta}(\sigma\cdot): L^p(\sigma dV_\alpha)\rightarrow L^q(\omega^q dV_\alpha),\,\,\,\beta\in \{0,1/3\},\,\,\,\sigma=\omega^{-p'}.
\end{equation}

Let $f\in L^p(\sigma dV_\alpha)$ and $g\in L^{q'}(\omega^q dV_\alpha)$ with $f,g>0$. We aim to estimate $$\langle Q_{\alpha,a}^{\beta}(\sigma f),g\omega^q \rangle_\alpha=\int_{\mathcal H}Q_{\alpha,a}^{\beta}(\sigma f)g\omega^q dV_\alpha.$$
We start with the case $\frac{p'}{q}\leq 1-\frac{a}{2+\alpha}$. We put $u=\omega^q$, and observe that $$[\omega]_{B_{p,q,\alpha}}=\sup_{I\subset \mathbb R}\frac{|Q_I|_{u,\alpha}}{|Q_I|_\alpha}\left(\frac{|Q_I|_{\sigma,\alpha}}{|Q_I|_{\alpha}}\right)^{\frac{q}{p'}}.$$

We use the notations $$B_{\sigma,\alpha}(f,Q_I)=\frac{1}{|Q_I|_{\sigma,\alpha}}\int_{Q_I}f\sigma dV_\alpha$$ and
$$B_{u,\alpha,a}(g,Q_I)=\frac{1}{|Q_I|_{u,\alpha}^{1-\frac{a}{2+\alpha}}}\int_{Q_I}gu dV_\alpha.$$
We obtain

\begin{eqnarray*}
\Pi &:=& \langle Q_{\alpha,a}^{\beta}(\sigma f),gu \rangle_\alpha\\ &=& \sum_{I\in \mathcal {D}^\beta}|I|^a\langle \sigma f,{\bf 1}_{Q_I} \rangle_\alpha\langle u g,{\bf 1}_{Q_I} \rangle_\alpha |Q_I|^{-1-\frac{\alpha}{2}}\\ &=& \sum_{I\in \mathcal {D}^\beta}\langle \sigma f,{\bf 1}_{Q_I} \rangle_\alpha\langle u g,{\bf 1}_{Q_I} \rangle_\alpha |Q_I|_\alpha^{-1+\frac{a}{2+\alpha}}\\
&=& \sum_{I\in \mathcal {D}^\beta}B_{\sigma,\alpha}(f,Q_I)B_{u,\alpha,a}(g,Q_I)\frac{|Q_I|_{\sigma,\alpha}|Q_I|_{u,\alpha}^{1-\frac{a}{2+\alpha}}}{|Q_I|_{\alpha}^{1-\frac{a}{2+\alpha}}}\\ &\le&
[\omega]_{B_{p,q,\alpha}}^{1-\frac{a}{2+\alpha}}\sum_{I\in \mathcal{D}^\beta}B_{\sigma,\alpha}(f,Q_I)B_{u,\alpha,a}(g,Q_I)\frac{|Q_I|_{\sigma,\alpha}|Q_I|_{u,\alpha}^{1-\frac{a}{2+\alpha}}}{|Q_I|_{\alpha}^{1-\frac{a}{2+\alpha}}}\times\\ && \left(\frac{|Q_I|_\alpha}{|Q_I|_{u,\alpha}}\right)^{1-\frac{a}{2+\alpha}}\left(\frac{|Q_I|_\alpha}{|Q_I|_{\sigma,\alpha}}\right)^{\frac{q}{p'}(1-\frac{a}{2+\alpha})}\\ &=& [\omega]_{B_{p,q,\alpha}}^{1-\frac{a}{2+\alpha}}\sum_{I\in \mathcal{D}^\beta}B_{\sigma,\alpha}(f,Q_I)B_{u,\alpha,a}(g,Q_I)|Q_I|_{\alpha}^{\frac{q}{p'}(1-\frac{a}{2+\alpha})}|Q_I|_{\sigma,\alpha}^{1-\frac{q}{p'}(1-\frac{a}{2+\alpha})}.
\end{eqnarray*}
 As $T_I\subset Q_I$ and $1-\frac{q}{p'}(1-\frac{a}{2+\alpha})\le 0$, $|Q_I|_{\sigma,\alpha}^{1-\frac{q}{p'}(1-\frac{a}{2+\alpha})}\lesssim |T_I|_{\sigma,\alpha}^{1-\frac{q}{p'}(1-\frac{a}{2+\alpha})}$. We remark that $u^{\frac{1}{(1-\frac{a}{2+\alpha})q}}\sigma^{\frac{1}{(1-\frac{a}{2+\alpha})p'}}=1$ and $\frac{1}{(1-\frac{a}{2+\alpha})q}+\frac{1}{(1-\frac{a}{2+\alpha})p'}=1$. It follows that 
 \Beas
 |T_I|_\alpha &=& \int_{T_I}u^{\frac{1}{(1-\frac{a}{2+\alpha})q}}\sigma^{\frac{1}{(1-\frac{a}{2+\alpha})p'}}dV_\alpha\\ &\leq & |T_I|_{u,\alpha}^{\frac{1}{(1-\frac{a}{2+\alpha})q}}|T_I|_{\sigma,\alpha}^{\frac{1}{(1-\frac{a}{2+\alpha})p'}}.
 \Eeas
 Hence as $|Q_I|_\alpha\backsimeq |T_I|_\alpha$, we deduce that
 $$|Q_I|_\alpha\lesssim |T_I|_{u,\alpha}^{\frac{1}{(1-\frac{a}{2+\alpha})q}}|T_I|_{\sigma,\alpha}^{\frac{1}{(1-\frac{a}{2+\alpha})p'}}.$$
It follows that
\Beas
\Pi &\lesssim& [\omega]_{B_{p,q,\alpha}}^{1-\frac{a}{2+\alpha}}\sum_{I\in \mathcal{D}^\beta}B_{\sigma,\alpha}(f,Q_I)B_{u,\alpha,a}(g,Q_I)|Q_I|_{\alpha}^{\frac{q}{p'}(1-\frac{a}{2+\alpha})}|Q_I|_{\sigma,\alpha}^{1-\frac{q}{p'}(1-\frac{a}{2+\alpha})}\\ &\lesssim& [\omega]_{B_{p,q,\alpha}}^{1-\frac{a}{2+\alpha}}\sum_{I\in \mathcal{D}^\beta}B_{\sigma,\alpha}(f,Q_I)B_{u,\alpha,a}(g,Q_I)|T_I|_{u,\alpha}^{\frac{1}{p'}}|T_I|_{\sigma,\alpha}^{\frac{1}{p}}\\ &\le& [\omega]_{B_{p,q,\alpha}}^{1-\frac{a}{2+\alpha}}\left(\sum_{I\in \mathcal {D}^\beta}|T_I|_{\sigma,\alpha}\left(B_{\sigma,\alpha}(f,Q_I)\right)^p\right)^{\frac{1}{p}}\left(\sum_{I\in \mathcal {D}^\beta}|T_I|_{\omega,\alpha}\left(B_{u,\alpha,a}(g,Q_I)\right)^{p'}\right)^{\frac{1}{p'}}\\ &=& [\omega]_{B_{p,q,\alpha}}^{1-\frac{a}{2+\alpha}}\left(\sum_{I\in \mathcal {D}^\beta}\int_{T_I}\left(B_{\sigma,\alpha}(f,Q_I)\right)^p\sigma dV_\alpha\right)^{\frac{1}{p}}\left(\sum_{I\in \mathcal {D}^\beta}\int_{T_I}\left(B_{u,\alpha,a}(g,Q_I)\right)^{p'}\omega dV_\alpha\right)^{\frac{1}{p'}}\\ &\le& [\omega]_{B_{p,q,\alpha}}^{1-\frac{a}{2+\alpha}}\|\mathcal{M}_{\sigma,\alpha}^\beta f\|_{L^p(\sigma dV_\alpha)}\|\mathcal{M}_{u,\alpha,a}^\beta g\|_{L^{p'}(u dV_\alpha)}\\ &\le& C_{p,q,\alpha}[\omega]_{B_{p,q,\alpha}}^{1-\frac{a}{2+\alpha}}\| f\|_{L^p(\sigma dV_\alpha)}\|g\|_{L^{q'}(u dV_\alpha)}.
\Eeas

For the case $\frac{q}{p'}\leq 1-\frac{a}{2+\alpha}$, we use the previous and duality. 
We observe that $Q_{\alpha,a}^\beta$ is self-adjoint with respect to the duality pairing $\langle \cdot,\cdot\rangle_\alpha$. Hence
\Beas \|Q_{\alpha,a}^\beta\|_{L^p(\omega^p dV_\alpha)\rightarrow L^q(\omega^q dV_\alpha)} &=& \|Q_{\alpha,a}^\beta\|_{L^{q'}(\omega^{-q'} dV_\alpha)\rightarrow L^{p'}(\omega^{-p'} dV_\alpha)}\\ &\leq& C_{p,q,\alpha}[\omega^{-1}]_{B_{q',p',\alpha}}^{1-\frac{a}{2+\alpha}}\\ &\leq& C_{p,q,\alpha}[\omega]_{B_{p,q,\alpha}}^{\frac{p'}{q}(1-\frac{a}{2+\alpha})}.
\Eeas
For the proof of the necessary part, we observe that $\mathcal{M}_{\alpha,a}f(z)\leq S_{\alpha,a}|f|(z)$ for any $z\in \mathcal{H} $. Put $f=\omega^{-p'}1_{Q_I}$ for some fixed interval $I\subset \mathbb{R}$. We recall the notations $u=\omega^q$ and $\sigma=\omega^{-p'}$. Observe that $$\mathcal{M}_{\alpha,a}f(z)\geq \frac{|Q_I|_{\sigma,\alpha}}{|Q_I|_{\alpha}^{1-\frac{a}{2+\alpha}}}.$$ Hence from the above observations and the boundedness of $S_{\alpha,a}$, we obtain that there is a positive constant $C>0$ such that
\Beas
C|Q_I|_{\sigma,\alpha}^{1/p}= C\left(\int_{Q_I}\omega^{-pp'+p}dV_\alpha\right)^{1/p} &=& C\|f\|_{L^p(\omega^p dV_\alpha)}\\ &\ge& \|\mathcal{M}_{\alpha,a}f\|_{L^q( udV_\alpha)}\\ &\ge& \frac{|Q_I|_{\sigma,\alpha}|Q_I|_{u,\alpha}^{1/q}}{|Q_I|_{\alpha}^{1-\frac{a}{2+\alpha}}}.
\Eeas
Thus 
\Beas [\omega]_{B_{p,q,\alpha}} &=& \sup_{I}\left(\frac{|Q_I|_{u,\alpha}}{|Q_I|_{\alpha}}\right)\left(\frac{|Q_I|_{\sigma,\alpha}}{|Q_I|_{\alpha}}\right)^{q/p'}\\ &=& \sup_{I}\frac{|Q_I|_{\sigma,\alpha}^{q/p'}|Q_I|_{u,\alpha}}{|Q_I|_{\alpha}^{q(1-\frac{a}{2+\alpha})}}\le C.
\Eeas
The proof is complete.
\end{proof}
\section{Example for sharpness}
We exhibit here an example of function and weight that show that our estimates are sharp. As this is a routine in this setting, we will go straight to the essential. We recall that $0\le a<2+\alpha$. Fix $0<\delta<1$. We assume that $\frac{p'}{q}\geq 1$. We consider $\omega(z)=|z|^{\frac{(2+\alpha-\delta)}{p'}}$. One easily checks that $\omega\in B_{p,q,\alpha}(\mathcal{H})$ and that $[\omega]_{B_{p,q,\alpha}}\simeq \delta^{-\frac{q}{p'}}$. We also consider the function $f(z)=|z|^{\delta-(2+\alpha)}{\bf 1}_{\{z\in \mathcal{H}:|z|\leq 1\}}(z)$. We obtain that $$\|f\|_{L^p(\mathcal{H}, \omega^pdV_\alpha)}\simeq \delta^{-1/p}.$$
\vskip .3cm
We recall the following facts from \cite{PR}: for each $\alpha>-1$, there exists a constant $M_\alpha>0$ such that for any $z\in \mathcal{H}$, $|z|\geq M_\alpha$, we have for any  $w_1,w_2\in  \{w\in \mathcal{H}: |w|\leq 1\}$, $$\arg((z-\bar{w}_1),(z-\bar{w}_2))\leq 2\arcsin\frac{1}{M_\alpha}.$$
Also, $|z-\bar{w}|^{2+\alpha}\leq 2^{2+\alpha}|z|^{2+\alpha}$, for $w\in \{w\in \mathcal{H}: |w|\leq 1\}$.
\vskip .3cm
It follows from the above facts that for any $z\in \mathcal{H}$ with $|z|\geq M_\alpha$,
\Beas
S_{\alpha,a}f(z=x+iy) &\geq& y^a\int_{\{w\in \mathcal{H}: |w|\leq 1\}}\frac{f(w)}{|z-\bar{w}|^{2+\alpha}}dV_\alpha(w)\\ &\geq& C_\alpha y^a|z|^{-(2+\alpha)}\int_{\{w\in \mathcal{H}: |w|\leq 1\}}|w|^{\delta-(2+\alpha)}dV_\alpha(w)\\ &\geq& \frac{C_{\alpha}}{\delta} y^a|z|^{-(2+\alpha)}.
\Eeas
Hence
\Beas
\|S_{\alpha,a}f\|_{L^q(\mathcal{H}, \omega^qdV_\alpha)}^q &\geq& C_{\alpha}\delta^{-q}\int_{|z|\geq M_\alpha}|z|^{-(2+\alpha)q+(2+\alpha-\delta)\frac{q}{p'}}dV_{\alpha+aq}(z)\\ &\geq& C_{p,q,\alpha}\delta^{-q-1}.
\Eeas
That is 
$$\|S_{\alpha,a}f\|_{L^q(\mathcal{H}, \omega^qdV_\alpha)}\gtrsim \delta^{-\frac{1}{q}-1}=\delta^{-(\frac{1}{p'}+\frac{1}{q})}\delta^{-\frac{1}{p}}\simeq [\omega]_{B_{p,q,\alpha}}^{(1-\frac{a}{2+\alpha})\frac{p'}{q}}\|f\|_{L^p(\mathcal{H}, \omega^pdV_\alpha)}.$$
\section{A sibling of the operator $S_{\alpha,a}$}
The following operator also generalizes the positive Bergman operator:
$$T_{\alpha,a}f(z)=\int_{\mathcal {H}}\frac{f(z)}{|z-\overline{w}|^{2+\alpha-a}}dV_{\alpha}(z)$$
where $0\leq a<2+\alpha$. We observe that $$S_{\alpha,a}f\le T_{\alpha,a}f\,\,\,\textrm{for all}\,\,\,f\geq 0.$$
The proof of the following result follows as above.
\begin{theorem}\label{thm:Taalphapq}
Let $\alpha>-1$, and $0\leq a<2+\alpha$. Let $1<p\leq q<\infty$, so that $\frac{1}{p}-\frac{1}{q}=\frac{a}{2+\alpha}$ and $\min\{\frac{p'}{q},\frac{q}{p'}\}\leq 1-\frac{a}{2+\alpha}$. Then $T_{\alpha,a}$ is bounded from $L^p(\mathcal{H}, \omega^pdV_\alpha)$ into $L^q(\mathcal{H}, \omega^qdV_{\alpha})$ if and only if $\omega\in B_{p,q,\alpha}(\mathcal{H})$. Moreover,
$$\|T_{\alpha,a}\|_{L^p(\mathcal{H},\omega^pdV_\alpha)\rightarrow L^q(\mathcal{H},\omega^qdV_\alpha)}\le C_{p,q,\alpha}[\omega]_{B_{p,q,\alpha}}^{(1-\frac{a}{2+\alpha})\max\{1,\frac{p'}{q}\}}.$$
Moreover, the power of $[\omega]_{B_{p,q,\alpha}}$ in the above inequality is sharp.
\end{theorem}
One can see the operators $S_{\alpha,a}$ and $T_{\alpha,a}$ as analogues of fractional operators of real harmonic analysis. Our results can then be viewed as analogues of the ones in \cite{Laceyetal, Moen}.
\section{Proof of Theorem \ref{thm:Bergpq}}
We start by recalling the following easy fact.
\begin{lemma}\label{lem:integkernel} Let $\alpha$ be real. Then
 the function
$f(z)=\left(\frac{z+it}{i}\right)^{-\alpha}$, with $t>0$, belongs to
$L^{p}(\mathcal{H}, dV_\nu)$, if and only if $\nu>-1$ and $\alpha >
\frac{\nu+2}{p}$. In this
case,$$||f||_{p,\nu}^p=C_{\alpha,p,q}t^{-p\alpha+\nu+2}.$$
\end{lemma}

We also recall the following off-diagonal Schur-type test.
\begin{lemma}[G. O. Okokiolu \cite{Oki}]\label{lem:okikiolu}
Let $p,r,q$ be positive numbers such that $1<p\le r$ and $\frac{1}{p} + \frac{1}{q} = 1.$ Let $K(x,y)$ be a complex-valued function measurable on $X\times Y$ and suppose there exist $0<t\le 1$, measurable functions $\varphi_1:X\rightarrow (0,\infty),\quad\varphi_2:Y\rightarrow (0,\infty)$ and nonnegative constants $M_1,M_2$ such that
\begin{eqnarray}
\label{ee1} 
\int_X\left|K(x,y)\right|^{tq}\varphi_1^{q}(y)\mathrm{d}\mu(y) &\le & M_1^{q}\varphi_2^{q}(x)\qquad\mbox{a.e on}\quad Y\quad\mbox{and}\\ 
\label{ee2}
\int_Y\left|K(x,y)\right|^{(1-t)r}\varphi_2^r(x)\mathrm{d}\nu(x) &\le  &M_2^r\varphi_1^r(y)\qquad\mbox{a.e on}\quad X.
\end{eqnarray} 
If $T$ is given by $$Tf(x)=\int_X\!\!\!f(y)K(x,y)\mathrm{d}\mu(y)$$ where $f\in L^p(X,\mathrm{d}\mu),$ then $T:L^p(X,\mathrm{d}\mu)\longrightarrow L^r(Y,\mathrm{d}\nu)$ is bounded and for each $f\in L^p(X,\mathrm{d}\mu)$, $$\left\|Tf\right\|_{L^r(Y,\mathrm{d}\nu)}\le M_1M_2\|f\|_{L^p(X,\mathrm{d}\mu)}.$$
\end{lemma}

\begin{proof}[Proof of Theorem \ref{thm:Bergpq}]
As said above the proof of the necessary part follows as in \cite{BanSeh}. Hence we only present the proof of the sufficiency.
\vskip .3cm
We recall that $b = a+1-\frac{\alpha+2}{p} + \frac{\beta+2}{q}$.  Let us put $\omega = a-b-\alpha-1$ and observe that $$\omega = a-b-\alpha-1 = -\left(\frac{\alpha+2}{p'}+\frac{\beta+2}{q} \right)<0.$$
Note that $\alpha+1<p(a+1)$ is equivalent to $(a-\alpha)+\frac{\alpha+1}{p'}>0.$ As $\omega<0,$ we obtain $(a-\alpha)\omega+\frac{\alpha+1}{p'}\omega <0,$ which is the same as $$-\frac{a -\alpha}{p'}(\alpha+2)-\frac{a-\alpha}{q}(\beta+2)+\frac{\alpha+1}{p'}\omega<0$$ or 
\begin{eqnarray}
\label{e1}
\frac{\alpha+1}{p'}\omega - \frac{a-\alpha}{p'}(\alpha+2)<\frac{a-\alpha}{q}(\beta+2).
\end{eqnarray}
We also have that $0<\frac{\beta+1}{q}$. 
From \eqref{e1} and the last observation, we have that one can find two numbers $r$ and $s$ with $r-s>0$, such that
\begin{eqnarray}
\label{e3}
\frac{\alpha+1}{p'}\omega - \frac{a-\alpha}{p'}(\alpha+2)<\omega s+(a-\alpha)(r-s)<\frac{a-\alpha}{q}(\beta+2)
\end{eqnarray}
and
\begin{eqnarray}
\label{e4}
0<r<\frac{\beta+1}{q}.
\end{eqnarray}
\eqref{e3} is equivalent to the inequality
\begin{equation}\label{e5}
-\frac{a-\alpha}{\omega}\left[-\frac{\beta+2}{q}+r-s\right] < s < \frac{\alpha+1}{p'}+\frac{a-\alpha}{\omega}\left[-\frac{\alpha+2}{p'}+s-r\right].
\end{equation}

Let $$t:=\frac{-\frac{\alpha+2}{p'}+s-r}{\omega}$$ so that $$1-t=\frac{r-s-\frac{\beta+2}{q}}{\omega}.$$ Using the definition of $t$, \eqref{e5} becomes
\begin{eqnarray}
\label{e7}
-(a-\alpha)(1-t)<s<\frac{\alpha+1}{p'}+(a-\alpha)t.
\end{eqnarray}

We now observe that the operator $T^+$ can be represented as
$$T^{+}f(z)=\int_{\mathcal{H}}f(w)K(z,w) dV_\alpha(w).$$
where $K(z,w)=\frac{(\Im w)^{a-\alpha}}{|z-\bar{w}|^{1+b}}$. Let us define  $$\varphi_1(w) = (\Im w)^{-s}\,\,\, \textrm{and}\,\,\,\varphi_2(w) = (\Im w)^{-r}.$$ Applying Okikiolu's test to $T^+$ we obtain at the first step 
\begin{eqnarray*}
\int_{\mathcal{H}}\!\!\!K(z,w)^{tp'}\varphi_1^{p'}(w)(\Im w)^\alpha \mathrm{d}V(w)& = &\int_{\mathcal{H}}\!\frac{v^{(a-\alpha)tp'}v^{-sp'+\alpha}}{|z-\overline{w}|^{t(1+b) p'}}\mathrm{d}V(w).
\end{eqnarray*}
From the right inequality in \eqref{e7} we have $a+1+(\beta-a)tp'-sp'>0$. Using the definition of $\omega, t$, we obtain
\begin{eqnarray*}
t(1+b) p' + sp' - (a-\alpha)tp'-\alpha-2 & = & (1+b-a+\alpha)tp'+sp'-\alpha-2\\
& = & \left(\alpha+2-\frac{\alpha+2}{p}+\frac{\beta+2}{q}\right)tp'+sp'-\alpha-2\\
& = & -\omega tp'+sp'-\alpha-2\\
& = & \left(\frac{\alpha+2}{p'}+r-s\right)p'+sp'-\alpha-2\\
& = & rp'>0.
\end{eqnarray*}
Hence we obtain from the above observations and Lemma \ref{lem:integkernel} that
$$\int_{\mathcal{H}}\!\!\!K(z,w)^{tp'}\varphi_1^{p'}(w)(\Im w)^\alpha \mathrm{d}V(w)=Cy^{-rp'}=C\varphi_2(z)^{p'}.$$
On the other hand, we have 
\begin{eqnarray*}
\int_{\mathcal{H}}\!\!\!K(z,w)^{(1-t)q}\varphi_2^{q}(z)(\Im z)^\beta \mathrm{d}V(z)& = &\int_{\mathcal{H}}\!\frac{(\Im w)^{(a-\alpha)(1-t)q}(\Im z)^{-rq'+\beta}}{|z-\overline{w}|^{(1-t)(1+b)q}}\mathrm{d}V(z)\\ &=& (\Im w)^{(a-\alpha)(1-t)q}\int_{\mathcal{H}}\!\frac{(\Im z)^{-rq+\beta}}{|z-\overline{w}|^{(1-t)(1+b)q}}\mathrm{d}V(z).
\end{eqnarray*}
From the inequality \eqref{e4}, we have $-rq+\beta+1>0$. From the definition of $\omega$ and $1-t,$ and the first inequality in \eqref{e7}, we obtain
\begin{eqnarray*}
(1+b)(1-t)q+rq-\beta-2 & = & (a-\alpha+\frac{\alpha+2}{p'}+\frac{\beta+2}{q})(1-t)q+rq-\beta-2\\
& = & (a-\alpha-\omega)(1-t)q+rq-\beta-2\\
& =& (a-\alpha)(1-t)q-\omega(1-t)q+rq-\beta-2\\
& = & (a-\alpha)(1-t)q-(\frac{\beta+2}{q}+s-r)q+rq-\beta-2\\
& = & (a-\alpha)(1-t)q+sq\\
& = & q[(a-\alpha)(1-t)+s]>0.
\end{eqnarray*}
Hence we obtain from the above observations and Lemma \ref{lem:integkernel} that
$$\int_{\mathcal{H}}\!\!\!K(z,w)^{(1-t)q}\varphi_2^{q}(z)(\Im z)^\beta \mathrm{d}V(w)=C(\Im w)^{-sq}=C\varphi_1(w)^{q}.$$
The proof is complete.

\end{proof}
\bibliographystyle{elsarticle-num}

\end{document}